\documentclass[11pt,leqno]{article}

\usepackage[active]{srcltx}

\usepackage{amsfonts,latexsym,amsmath,amssymb,amsthm}
\usepackage{fullpage}
\newtheorem{theorem}{Theorem}[section]
\newtheorem{lemma}[theorem]{Lemma}

\newtheorem{proposition}[theorem]{Proposition}

\newtheorem{remark}[theorem]{Remark}

\theoremstyle{definition}

\theoremstyle{remark}

\newtheorem*{note*}{Note}

\numberwithin{equation}{section}

\makeatletter
\newcommand{\rank}{\mathop{\operator@font rank}}
\newcommand{\conv}{\mathop{\operator@font conv}}
\newcommand{\vol}{\mathop{\operator@font vol}}
\newcommand{\onetagright}{\tagsleft@false}
\makeatother

\newcommand{\ls}{\leqslant}
\newcommand{\gr}{\geqslant}
\renewcommand{\epsilon}{\varepsilon}
\newcommand{\prend}{$\quad \hfill \Box$}
\usepackage{color}

\usepackage{hyperref}

\begin{document}
\small

\title{\bf Estimates for measures of lower dimensional sections of convex bodies}

\medskip

\author{Giorgos Chasapis, Apostolos Giannopoulos and Dimitris-Marios Liakopoulos}

\date{}

\maketitle

\begin{abstract}
\footnotesize We present an alternative approach to some results of Koldobsky on measures of sections of symmetric
convex bodies, which allows us to extend them to the not necessarily symmetric setting. We prove that if $K$ is a convex body in ${\mathbb R}^n$ with
$0\in {\rm int}(K)$ and if $\mu $ is a measure on ${\mathbb R}^n$ with a locally integrable
non-negative density $g$ on ${\mathbb R}^n$, then
\begin{equation*}\mu (K)\ls \left (c\sqrt{n-k}\right )^k\max_{F\in G_{n,n-k}}\mu (K\cap F)\cdot |K|^{\frac{k}{n}}\end{equation*}
for every $1\ls k\ls n-1$. Also, if $\mu $ is even and log-concave, and if $K$ is a symmetric convex body in ${\mathbb R}^n$
and $D$ is a compact subset of ${\mathbb R}^n$ such that
$\mu (K\cap F)\ls \mu (D\cap F)$ for all $F\in G_{n,n-k}$, then
\begin{equation*}\mu (K)\ls \left (ckL_{n-k}\right )^{k}\mu (D),\end{equation*}
where $L_s$ is the maximal isotropic constant of a convex body in ${\mathbb R}^s$.
Our method employs a generalized Blaschke-Petkantschin formula and estimates
for the dual affine quermassintegrals.
\end{abstract}

\section{Introduction}

In this article we discuss lower dimensional versions of the slicing problem and of the Busemann-Petty problem, both in the classical setting and in the generalized setting of arbitrary measures in place of volume, which was put forward by Koldobsky for the slicing problem and by Zvavitch for the Busemann-Petty problem.
We introduce an alternative approach which is based on the generalized Blaschke-Petkantschin formula and on asymptotic estimates for the dual affine quermassintegrals.

The classical slicing problem asks if there exists an absolute constant $C_1>0$ such that for every $n\gr 1$ and every
convex body $K$ in ${\mathbb R}^n$ with center of mass at the origin (we call these convex bodies centered)
one has
\begin{equation}\label{eq:intro-1}|K|^{\frac{n-1}{n}}\ls C_1\,\max_{\theta\in S^{n-1}}\,|K\cap \theta^{\perp }|.\end{equation}
It is well-known that this problem is equivalent to the question if there exists an absolute constant $C_2>0$ such that
\begin{equation}\label{eq:intro-2}L_n:= \max\{ L_K:K\ \hbox{is isotropic in}\ {\mathbb R}^n\}\ls C_2\end{equation}
for all $n\gr 1$ (see Section 2 for background information on isotropic convex bodies and log-concave probability
measures). Bourgain proved in \cite{Bourgain-1991} that $L_n\ls
c\sqrt[4]{n}\log\! n$, and Klartag \cite{Klartag-2006} improved this bound to $L_n\ls c\sqrt[4]{n}$. A second proof of Klartag's bound
appears in \cite{Klartag-EMilman-2012}. From the equivalence of the two questions it follows that
\begin{equation}\label{eq:intro-3}|K|^{\frac{n-1}{n}}\ls c_1L_n\,\max_{\theta\in S^{n-1}}\,|K\cap \theta^{\perp }|
\ls c_2\sqrt[4]{n}\,\max_{\theta\in S^{n-1}}\,|K\cap \theta^{\perp }|\end{equation}
for every centered convex body $K$ in ${\mathbb R}^n$.

The natural generalization, the lower dimensional slicing problem, is the following question: Let $1\ls k\ls n-1$ and let $\alpha_{n,k}$
be the smallest positive constant $\alpha >0$ with the following property: For every centered convex body $K$ in ${\mathbb R}^n$ one has
\begin{equation}\label{eq:intro-4}|K|^{\frac{n-k}{n}}\ls \alpha^k\max_{F\in G_{n,n-k}}|K\cap F|.\end{equation}
\begin{quote}\label{prob:low-dim-slicing}{\sl Is it true that there exists an absolute constant $C_3>0$ such that $\alpha_{n,k}\ls C_3$ for all $n$ and $k$?}
\end{quote}From \eqref{eq:intro-3} we have $\alpha_{n,1}\ls cL_n$ for an absolute constant $c>0$. We also restrict the question to the
class of symmetric convex bodies and denote the corresponding constant by $\alpha_{n,k}^{(s)}$.

The problem can be posed for a general measure in place of volume. Let $g$ be a locally integrable non-negative function on ${\mathbb R}^n$.
For every Borel subset $B\subseteq {\mathbb R}^n$ we define
\begin{equation}\mu (B)=\int_Bg(x)dx,\end{equation}
where, if $B\subseteq F$ for some subspace $F\in G_{n,s}$, $1\ls s\ls n-1$, integration is understood with respect to the $s$-dimensional Lebesgue
measure on $F$. Then, for any $1\ls k\ls n-1$ one may define $\alpha_{n,k}(\mu )$ as the smallest constant $\alpha >0$ with the following property: For
every centered convex body $K$ in ${\mathbb R}^n$ one has
\begin{equation}\label{eq:intro-5}\mu (K)\ls \alpha^k\max_{F\in G_{n,n-k}}\mu (K\cap F)\,|K|^{\frac{k}{n}}.\end{equation}
Koldobsky proved in \cite{Koldobsky-Advances-2014} that if $K$ is a symmetric convex body in ${\mathbb R}^n$ and if $g$ is even and continuous on $K$ then
\begin{equation}\label{eq:intro-6}\mu (K)\ls \gamma_{n,1}\frac{n}{n-1}\sqrt{n}\max_{\theta\in S^{n-1}}\mu (K\cap \theta^{\perp })\,|K|^{\frac{1}{n}},\end{equation}
where, more generally, $\gamma_{n,k}=|B_2^n|^{\frac{n-k}{n}}/|B_2^{n-k}|<1$ for all $1\ls k\ls n-1$.
In other words, for the symmetric (both with respect to $\mu $ and $K$) analogue $\alpha_{n,1}^{(s)}$ of $\alpha_{n,1}$ one has
\begin{equation}\label{eq:intro-7}\sup_{\mu }\alpha_{n,1}^{(s)}(\mu )\ls c_3\sqrt{n}.\end{equation}
In \cite{Koldobsky-GAFA-2014}, Koldobsky obtained estimates for the lower dimensional sections: if $K$ is a symmetric convex body in ${\mathbb R}^n$ and if $g$ is even
and continuous on $K$ then
\begin{equation}\label{eq:intro-8}\mu (K)\ls \gamma_{n,k}\frac{n}{n-k}(\sqrt{n})^k\max_{F\in G_{n,n-k}}\mu (K\cap F)\,|K|^{\frac{k}{n}}\end{equation}
for every $1\ls k\ls n-1$. In other words, for the symmetric analogue $\alpha_{n,k}^{(s)}$ of $\alpha_{n,k}$ one has
\begin{equation}\label{eq:intro-9}\sup_{\mu }\alpha_{n,k}^{(s)}(\mu )\ls c_4\sqrt{n}.\end{equation}
We provide a different proof of this fact; our method allows us to drop the symmetry and continuity assumptions.

\begin{theorem}\label{th:intro-arb-1}Let $K$ be a convex body in ${\mathbb R}^n$ with $0\in {\rm int}(K)$. Let $g$ be a bounded locally integrable
non-negative function on ${\mathbb R}^n$ and let $\mu $ be the measure on ${\mathbb R}^n$ with density $g$. For every $1\ls k\ls n-1$,
\begin{equation}\mu (K)\ls \left (c_5\sqrt{n-k}\right )^k\max_{F\in G_{n,n-k}}\mu (K\cap F)\cdot |K|^{\frac{k}{n}},\end{equation}
where $c_5>0$ is an absolute constant. In particular, $\alpha_{n,k}(\mu )\ls c_5\sqrt{n-k}$.
\end{theorem}

The classical Busemann-Petty problem is the following question. Let $K$ and $D$ be two origin-symmetric convex bodies
in ${\mathbb R}^n$ such that
\begin{equation}\label{eq:intro-10}|K\cap\theta^{\perp }|\ls |D\cap\theta^{\perp }|\end{equation}
for all $\theta\in S^{n-1}$. Does it follow that $|K|\ls |D|$? The answer is affirmative if $n\ls 4$ and negative if $n\gr 5$
(for the history and the solution to this problem, see Koldobsky's monograph \cite{Koldobsky-book}). The isomorphic version of the Busemann-Petty
problem asks if there exists an absolute constant $C_4>0$ such that whenever $K$ and $D$ satisfy \eqref{eq:intro-10} we have $|K|\ls C_4|D|$.
This question is equivalent to the slicing problem and to the isotropic constant conjecture (asking if $\{L_n\}$ is a bounded sequence).
More precisely, it is known that if $K$ and $D$ are two centered convex bodies
in ${\mathbb R}^n$ such that \eqref{eq:intro-10} holds true for all $\theta\in S^{n-1}$, then
\begin{equation}\label{eq:intro-11}|K|^{\frac{n-1}{n}}\ls c_6L_n\,|D|^{\frac{n-1}{n}},\end{equation}
where $c_6>0$ is an absolute constant.

The natural generalization, the lower dimensional Busemann-Petty problem, is the following question: Let $1\ls k\ls n-1$
and let $\beta_{n,k}$ be the smallest constant $\beta >0$ with the following property: For
every pair of centered convex bodies $K$ and $D$ in ${\mathbb R}^n$ that satisfy
\begin{equation}\label{eq:intro-12}|K\cap F|\ls |D\cap F|\end{equation}
for all $F\in G_{n,n-k}$, one has
\begin{equation}\label{eq:intro-13}|K|^{\frac{n-k}{n}}\ls \beta^k\,|D|^{\frac{n-k}{n}}.\end{equation}
\begin{quote}\label{prob:low-dim-BP}{\sl Is it true that there exists an absolute constant $C_5>0$ such that $\beta_{n,k}\ls C_5$ for all $n$ and $k$?}\end{quote}
From \eqref{eq:intro-11} we have $\beta_{n,1}\ls c_6L_n\ls c_7\sqrt[4]{n}$ for some absolute constant $c_7>0$. We also consider the same question for the
class of symmetric convex bodies and we denote the corresponding constant by $\beta_{n,k}^{(s)}$.

As in the case of the slicing problem, the same question can be posed for a general measure in place of volume.
For any $1\ls k\ls n-1$ and any measure $\mu $ on ${\mathbb R}^n$ with a locally integrable non-negative density $g$
one may define $\beta_{n,k}(\mu )$ as the smallest constant $\beta >0$ with the following property: For
every pair of centered convex bodies $K$ and $D$ in ${\mathbb R}^n$ that satisfy $\mu (K\cap F)\ls \mu (D\cap F)$
for every $F\in G_{n,n-k}$, one has
\begin{equation}\label{eq:intro-55}\mu (K)\ls \beta^k\mu (D).\end{equation}
Similarly, one may define the ``symmetric" constant $\beta_{n,k}^{(s)}(\mu )$.
Koldobsky and Zvavitch \cite{Koldobsky-Zvavitch-2015} proved that $\beta_{n,1}^{(s)}(\mu )\ls \sqrt{n}$
for every measure $\mu $ with an even continuous non-negative density. In fact, the study of these questions in the
setting of general measures was initiated by Zvavitch in \cite{Zvavitch-2005}, where he proved
that the classical Busemann-Petty problem for general measures has an affirmative answer if $n\ls 4$ and
a negative one if $n\gr 5$. We study the lower dimensional question and
provide a general estimate in the case where $\mu $ has an even log-concave density.

\begin{theorem}\label{th:intro-arb-2}Let $\mu $ be a measure on ${\mathbb R}^n$ with an even log-concave density $g$ and let $1\ls k\ls n-1$.
Let $K$ be a symmetric convex body in ${\mathbb R}^n$ and let $D$ be a compact subset of ${\mathbb R}^n$ such that
\begin{equation}\label{eq:intro-14}\mu (K\cap F)\ls \mu (D\cap F)\end{equation}
for all $F\in G_{n,n-k}$. Then,
\begin{equation}\label{eq:intro-15}\mu (K)\ls \left (c_8kL_{n-k}\right )^{k}\mu (D),\end{equation}
where $c_8>0$ is an absolute constant.
\end{theorem}

We prove Theorem \ref{th:intro-arb-1} and Theorem \ref{th:intro-arb-2} in Section 4. Our main tools are the generalized Blaschke-Petkantschin formula and Grinberg's inequality for the dual affine quermassintegrals of a convex body. For the proof of Theorem \ref{th:intro-arb-2} we also use a recent result of Dann, Paouris and Pivovarov. We introduce
these results in Section 3.

\smallskip

In Section 5 we collect some results for the case of volume; we obtain the following bounds for the constants $\alpha_{n,k}$ and $\beta_{n,k}$.

\begin{theorem}\label{th:intro-vol}For every $1\ls k\ls n-1$ we have
\begin{equation}\label{eq:intro-16}\alpha_{n,k}\ls \overline{c}_1L_n,\end{equation}
where $\overline{c}_1>0$ is an absolute constant. Moreover, for codimensions $k$ which are proportional to $n$ we have the stronger
bound
\begin{equation}\label{eq:intro-17}\alpha_{n,k}\ls \overline{c}_2\sqrt{n/k}\,(\log (en/k))^{\frac{3}{2}},\end{equation}
where $\overline{c}_2>0$ is an absolute constant. Finally,
\begin{equation}\label{eq:intro-18}\beta_{n,k}\ls \overline{c}_3L_n\end{equation}
where $\overline{c}_3>0$ is an absolute constant, and
\begin{equation}\label{eq:intro-19}\beta_{n,k}\ls \overline{c}_4\sqrt{n/k}\,(\log (en/k))^{\frac{3}{2}}\end{equation}
where $\overline{c}_4>0$ is an absolute constant. One also has
\begin{equation}\label{eq:intro-20}\alpha_{n,k}\ls \beta_{n,k}\end{equation}
for all $n$ and $k$.
\end{theorem}

Most of the estimates in Theorem \ref{th:intro-vol} are probably known to specialists; we just point out alternative ways to
justify them. In particular, Koldobsky has proved in \cite{Koldobsky-AAM-2015} that if $\lambda\in (0,1)$ and $k>\lambda n$ then
\begin{equation}\label{eq:intro-21}\beta_{n,k}^{(s)}\ls \overline{c}_4\sqrt{\frac{(1-\log\lambda )^3}{\lambda }},\end{equation}
where $\overline{c}_4>0$ is an absolute constant; this is the symmetric analogue of \eqref{eq:intro-19}. 
It should be also mentioned that Koldobsky has proved \eqref{eq:intro-17} for all
symmetric convex bodies $K$ and any even measure $\mu $ with a continuous, even and non-negative density $g$ (see Section 6
for a list of other related results).

We close this article with a general stability estimate in the spirit of Koldobsky's stability theorem (see Theorem \ref{th:koldobsky-1}).

\begin{theorem}\label{th:intro-stability}Let $1\ls k\ls n-1$ and let $K$ be a compact set in ${\mathbb R}^n$.
If $g$ is a locally integrable non-negative function on ${\mathbb R}^n$ such that
\begin{equation}\label{eq:intro-22}\int_{K\cap F}g(x)dx\ls\varepsilon \end{equation}
for some $\varepsilon >0$ and for all $F\in G_{n,n-k}$, then
\begin{equation}\label{eq:intro-23}\int_Kg(x)dx\ls \left (c_0\sqrt{n-k}\right )^k|K|^{\frac{k}{n}}\varepsilon ,\end{equation}
where $c_0>0$ is an absolute constant.
\end{theorem}

\section{Notation and preliminaries}

We work in ${\mathbb R}^n$, which is equipped with a Euclidean structure $\langle\cdot ,\cdot\rangle $. We denote the corresponding
Euclidean norm by $\|\cdot \|_2$, and write $B_2^n$ for the Euclidean unit ball, and $S^{n-1}$ for the unit sphere. Volume is
denoted by $|\cdot |$. We write $\omega_n$ for the volume of $B_2^n$ and $\sigma $ for the rotationally invariant probability measure on
$S^{n-1}$. We also denote the Haar measure on $O(n)$ by $\nu $. The Grassmann manifold $G_{n,k}$ of $k$-dimensional subspaces of
${\mathbb R}^n$ is equipped with the Haar probability measure $\nu_{n,k}$. Let $k\ls n$ and $F\in G_{n,k}$. We will denote the
orthogonal projection from $\mathbb R^{n}$ onto $F$ by $P_F$. We also define $B_F=B_2^n\cap F$ and $S_F=S^{n-1}\cap
F$.

The letters $c,c^{\prime }, c_1, c_2$ etc. denote absolute positive constants whose value may change from line to line. Whenever we
write $a\simeq b$, we mean that there exist absolute constants $c_1,c_2>0$ such that $c_1a\ls b\ls c_2a$.  Also if $K,L\subseteq
\mathbb R^n$ we will write $K\simeq L$ if there exist absolute constants $c_1, c_2>0$ such that $c_{1}K\subseteq L \subseteq
c_{2}K$.

\medskip

\noindent \textbf{Convex bodies.} A convex body in ${\mathbb R}^n$ is a compact convex subset $K$ of
${\mathbb R}^n$ with nonempty interior. We say that $K$ is symmetric if $K=-K$. We say that $K$ is centered if
the center of mass of $K$ is at the origin, i.e.~$\int_K\langle
x,\theta\rangle \,d x=0$ for every $\theta\in S^{n-1}$.

The volume radius of $K$ is the quantity ${\rm vrad}(K)=\left (|K|/|B_2^n|\right )^{1/n}$.
Integration in polar coordinates shows that if the origin is an interior point of $K$ then the volume radius of $K$ can be expressed as
\begin{equation}\label{eq:not-1}{\rm vrad}(K)=\left (\int_{S^{n-1}}\|\theta \|_K^{-n}\,d\sigma (\theta )\right)^{1/n},\end{equation}
where $\|\theta \|_K=\min\{ t>0:\theta \in tK\}$. The radial function of $K$ is defined by $\rho_K(\theta )=\max\{ t>0:t\theta\in K\}$,
$\theta\in S^{n-1}$. The support
function of $K$ is defined by $h_K(y):=\max \bigl\{\langle x,y\rangle :x\in K\bigr\}$, and
the mean width of $K$ is the average
\begin{equation}\label{eq:not-2}w(K):=\int_{S^{n-1}}h_K(\theta )\,d\sigma (\theta )\end{equation}
of $h_K$ on $S^{n-1}$. The radius $R(K)$ of $K$ is the smallest $R>0$ such that $K\subseteq RB_2^n$.
For notational convenience we write $\overline{K}$ for
the homothetic image of volume $1$ of a convex body $K\subseteq
\mathbb R^n$, i.e. $\overline{K}:= |K|^{-1/n}K$.

The polar body $K^{\circ }$ of a convex body $K$ in ${\mathbb R}^n$ with $0\in {\rm int}(K)$ is defined by
\begin{equation}\label{eq:not-3}
K^{\circ}:=\bigl\{y\in {\mathbb R}^n: \langle x,y\rangle \ls 1\;\hbox{for all}\; x\in K\bigr\}.
\end{equation}
The Blaschke-Santal\'{o} inequality states that if $K$ is centered then $|K||K^{\circ }|\ls |B_2^n|^2$,
with equality if and only if $K$ is an ellipsoid.
The reverse Santal\'{o} inequality of Bourgain and V. Milman states that there exists an absolute constant $c>0$ such
that, conversely,
\begin{equation}\label{eq:not-4}\left (|K||K^{\circ }|\right )^{1/n}\gr c/n\end{equation}
whenever $0\in {\rm int}(K)$. A convex body $K$ in ${\mathbb R}^n$ is called isotropic if it has volume $1$, it is centered, and if its inertia matrix is a multiple of the identity matrix:
there exists a constant $L_K >0$ such that
\begin{equation}\label{isotropic-condition}\int_K\langle x,\theta\rangle^2dx =L_K^2\end{equation}
for every $\theta $ in the Euclidean unit sphere $S^{n-1}$. For every centered convex body $K$ in ${\mathbb R}^n$
there exists an invertible linear transformation $T\in GL(n)$ such that $T(K)$ is isotropic. This isotropic image of $K$ is
uniquely determined up to orthogonal transformations.

For basic facts from the Brunn-Minkowski theory and the asymptotic theory of convex bodies we refer to the books \cite{Schneider-book} and \cite{AGA-book} respectively.

\smallskip

\noindent \textbf{Log-concave probability measures.}
We denote by ${\mathcal{P}}_n$ the class of all Borel probability measures on $\mathbb R^n$ which are absolutely
continuous with respect to the Lebesgue measure. The density of $\mu \in {\mathcal{P}}_n$ is denoted by $f_{\mu}$. We say that $\mu
\in {\mathcal{P}}_n$ is centered and we write $\textrm{bar}(\mu )=0$ if, for all $\theta\in S^{n-1}$,
\begin{equation}\label{eq:not-5}
\int_{\mathbb R^n} \langle x, \theta \rangle d\mu(x) = \int_{\mathbb
R^n} \langle x, \theta \rangle f_{\mu}(x) dx = 0.
\end{equation}
A measure $\mu$ on $\mathbb R^n$ is called log-concave if $\mu(\lambda
A+(1-\lambda)B) \gr \mu(A)^{\lambda}\mu(B)^{1-\lambda}$ for any compact subsets $A$
and $B$ of ${\mathbb R}^n$ and any $\lambda \in (0,1)$. A function
$f:\mathbb R^n \rightarrow [0,\infty)$ is called log-concave if
its support $\{f>0\}$ is a convex set and the restriction of $\log{f}$ to it is concave.
It is known that if a probability measure $\mu $ is log-concave and $\mu (H)<1$ for every
hyperplane $H$, then $\mu \in {\mathcal{P}}_n$ and its density
$f_{\mu}$ is log-concave. Note that if $K$ is a convex body in
$\mathbb R^n$ then the Brunn-Minkowski inequality implies that the indicator function
$\mathbf{1}_{K} $ of $K$ is the density of a log-concave measure.

If $\mu $ is a log-concave measure on ${\mathbb R}^n$ with density $f_{\mu}$, we define the isotropic constant of $\mu $ by
\begin{equation}\label{eq:definition-isotropic}
L_{\mu }:=\left (\frac{\sup_{x\in {\mathbb R}^n} f_{\mu} (x)}{\int_{{\mathbb
R}^n}f_{\mu}(x)dx}\right )^{\frac{1}{n}} [\det \textrm{Cov}(\mu)]^{\frac{1}{2n}},\end{equation} where
$\textrm{Cov}(\mu)$ is
the covariance matrix of $\mu$ with entries
\begin{equation}\label{eq:not-6}\textrm{Cov}(\mu )_{ij}:=\frac{\int_{{\mathbb R}^n}x_ix_j f_{\mu}
(x)\,dx}{\int_{{\mathbb R}^n} f_{\mu} (x)\,dx}-\frac{\int_{{\mathbb
R}^n}x_i f_{\mu} (x)\,dx}{\int_{{\mathbb R}^n} f_{\mu}
(x)\,dx}\frac{\int_{{\mathbb R}^n}x_j f_{\mu}
(x)\,dx}{\int_{{\mathbb R}^n} f_{\mu} (x)\,dx}.\end{equation} We say
that a log-concave probability measure $\mu $ on ${\mathbb R}^n$
is isotropic if $\textrm{bar}(\mu )=0$ and $\textrm{Cov}(\mu )$ is the identity matrix and
we write $\mathcal{IL}_n$ for the class of isotropic $\log $-concave probability measures on ${\mathbb R}^n$.
Note that a centered convex body $K$ of volume $1$ in ${\mathbb R}^n$ is isotropic,
i.e.~it satisfies (\ref{isotropic-condition}),
if and only if the log-concave probability measure $\mu_K$ with density
$x\mapsto L_K^n\mathbf{1}_{K/L_K}(x)$ is isotropic. We shall use the fact that for every log-concave measure $\mu $
on ${\mathbb R}^n$ one has
\begin{equation}\label{eq:Lmu}L_{\mu }\ls \kappa L_n,\end{equation}
where $\kappa >0$ is an absolute constant (a proof can be found in \cite[Proposition 2.5.12]{BGVV-book}).

Let $\mu\in {\mathcal P}_n$. For every $1\ls k\ls n-1$ and every
$E\in G_{n,k}$, the marginal of $\mu$ with respect to $E$ is the probability
measure $\pi_E(\mu )$ with density
\begin{equation}\label{definitionmarginal}f_{\pi_E(\mu )}(x)= \int_{x+
E^{\perp}} f_{\mu }(y) dy.
\end{equation}
It is easily checked that if $\mu $ is centered, isotropic or log-concave, then $\pi_E(\mu )$ is also centered, isotropic or
log-concave, respectively.

If $\mu$ is a measure on $\mathbb R^n$ which is absolutely continuous with respect
to the Lebesgue measure, and if $f_\mu$ is the density of $\mu$ and
$f_{\mu}(0) > 0$, then for every $p>0$ we define
\begin{equation}\label{eq:not-7}
K_p(\mu):=K_p(f_\mu)=\left\{x : \int_0^\infty r^{p-1}f_\mu(rx)\,
dr\gr \frac{f_{\mu }(0)}{p} \right\}.
\end{equation}
From the definition it follows that $K_p(\mu )$ is a star body with radial function
\begin{equation}\label{eq:not-8}
\rho_{K_p(\mu )}(x)=\left (\frac{1}{f_{\mu }(0)}\int_0^{\infty
}pr^{p-1}f_{\mu }(rx)\,dr\right )^{1/p}\end{equation} for $x\neq 0$.
The bodies $K_p(\mu )$ were introduced by K. Ball who showed that if $\mu $ is log-concave
then, for every $p>0$, $K_p(\mu )$ is a convex body.

For more information on isotropic convex bodies and log-concave measures see \cite{BGVV-book}.

\section{Tools from integral geometry and auxiliary estimates}

Our approach is based on the following generalized Blaschke-Petkantschin formula (see \cite[Chapter 7.2]{Schneider-Weil-book} and \cite[Lemma 5.1]{Gardner-2007}
for the particular case that we need):

\begin{lemma}\label{lem:gardner-07}Let $1\ls s\ls n-1$. There exists a constant $p(n,s)>0$ such that, for every non-negative
bounded Borel measurable function $f:({\mathbb R}^n)^s\to {\mathbb R}$,
\begin{align}\label{eq:tools-1}&\int_{{\mathbb R}^n}\cdots \int_{{\mathbb R}^n}f(x_1,\ldots ,x_s)dx_1\cdots dx_s\\
\nonumber &\hspace*{1cm} =p(n,s)\int_{G_{n,s}}\int_F\cdots \int_Ff(x_1,\ldots ,x_s)\,|{\rm conv}(0,x_1,\ldots ,x_s)|^{n-s}dx_1\ldots dx_s\,d\nu_{n,s}(F).
\end{align}
The exact value of the constant $p(n,s)$ is
\begin{equation}\label{eq:tools-2}p(n,s)=(s!)^{n-s}\frac{(n\omega_n)\cdots ((n-s+1)\omega_{n-s+1})}{(s\omega_s)\cdots (2\omega_2)\omega_1}.\end{equation}
\end{lemma}

Let $K$ be a compact set in ${\mathbb R}^n$. Applying Lemma \ref{lem:gardner-07} with $s=n-k$ for the function
$f(x_1,\ldots ,x_{n-k})=\prod_{i=1}^{n-k}{\bf 1}_{K}(x_i)$ we get
\begin{equation}\label{eq:tools-3}|K|^{n-k}=p(n,n-k)\int_{G_{n,n-k}}\int_{K\cap F}\cdots\int_{K\cap F}\,|{\rm conv}(0,x_1,\ldots ,x_{n-k})|^{k}dx_1\ldots dx_{n-k}\,d\nu_{n,n-k}(F).\end{equation}
We will use some basic facts about Sylvester-type functionals. Let $D$ be a convex body in ${\mathbb R}^m$. For every $p>0$ we consider
the normalized $p$-th moment of the expected volume of the random simplex ${\rm conv}(0,x_1,\ldots ,x_m)$, the convex hull of the origin and
$m$ points from $D$, defined by
\begin{equation}\label{eq:tools-4}S_p(D)=\left (\frac{1}{|D|^{m+p}}
\int_D\cdots\int_D|{\rm conv}(0,x_1,\ldots ,x_m)|^pdx_1\cdots
dx_m\right )^{1/p}.\end{equation}
Also, for any Borel probability measure $\nu $ on ${\mathbb R}^m$ we define
\begin{equation}\label{eq:tools-5}S_p(\nu )=\left (\int_{{\mathbb R}^m}\cdots\int_{{\mathbb R}^m}|{\rm conv}(0,x_1,\ldots ,x_m)|^pd\nu (x_1)\cdots
d\nu (x_m)\right )^{1/p}.\end{equation}
Note that $S_p(D)$ is invariant under invertible linear transformations: $S_p(D)=S_p(T(D))$ for every $T\in GL(n)$.
The next fact is well-known and goes back to Blaschke (see e.g. \cite[Proposition 3.5.5]{BGVV-book}).

\begin{lemma}\label{lem:tools-blaschke}Let $\nu $ be a centered Borel probability measure on ${\mathbb R}^m$. Then,
\begin{equation}\label{eq:tools-6}m!\,S_2^2(\nu )=\det ({\rm Cov}(\nu )).\end{equation}
In particular, if $D$ is centered then
\begin{equation}\label{eq:tools-7}S_2^2(D) =\frac{L_D^{2m}}{m!}.\end{equation}
\end{lemma}

H\"{o}lder's inequality shows that the function $p\mapsto S_p(D)$ is increasing on $(0,\infty )$. We will need
the next reverse H\"{o}lder inequality.

\begin{lemma}\label{lem:arb-mu-1}There exists an absolute constant $\delta >0$ such that, for every
log-concave probability measure $\nu $ on ${\mathbb R}^m$ and every $p>1$,
\begin{equation}\label{eq:tools-8}S_p(\nu )\ls (\delta p)^m\,S_1(\nu ).\end{equation}
In particular, for every convex body $D$ in ${\mathbb R}^m$ and every $p>1$,
\begin{equation}\label{eq:tools-9}S_p(D)\ls (\delta p)^m\,S_1(D).\end{equation}
\end{lemma}

\noindent {\it Proof.} We use the fact that there exists an absolute constant $\delta >0$ with the
following property: if $\nu \in {\cal P}_m$ is a log-concave probability measure then, for any seminorm $u:{\mathbb R}^m\to {\mathbb R}$
and any $q>p\gr 1$,
\begin{equation}\label{eq:tools-10}\left (\int_{{\mathbb R}^m}|u(x)|^qd\nu (x)\right )^{1/q}
\ls \frac{\delta q}{p}\left (\int_{{\mathbb R}^m}|u(x)|^pd\nu (x)\right )^{1/p}.\end{equation}
This is a consequence of Borell's lemma (see e.g. \cite[Theorem 2.4.6]{BGVV-book}). Next, recall that
\begin{equation}\label{eq:tools-11}|{\rm conv}(0,x_1,\ldots ,x_m)|=\frac{1}{m!}|\det (x_1,\ldots ,x_m)|.\end{equation}
The function $u_i:{\mathbb R}^m\to {\mathbb R}$ defined by $x_i\mapsto |{\rm det}(x_1,\ldots ,x_n)|$ for fixed $x_j$ in ${\mathbb R}^m$,
$j\neq i$, is a seminorm, as is the function $v_i:{\mathbb R}^m\to {\mathbb R}$
defined by
\begin{equation}\label{eq:tools-12}x_i\mapsto \int_{{\mathbb R}^m}\cdots\int_{{\mathbb R}^m} |{\rm det}(x_1,\ldots ,x_m)| dx_{i+1}\cdots dx_m\end{equation}
for fixed $x_j$ ($1\ls j<i$) in ${\mathbb R}^m$. By consecutive applications of
Fubini's theorem and of \eqref{eq:tools-10} we obtain \eqref{eq:tools-8}.
\prend

\medskip

The next lemma gives upper bounds for the constants $\gamma_{n,k}=|B_2^n|^{\frac{n-k}{n}}/|B_2^{n-k}|$ and $p(n,n-k)$; both constants appear frequently in the next sections.

\begin{lemma}\label{lem:estimate-constant}For every $1\ls k\ls n-1$ we have
\begin{equation}\label{eq:tools-13}e^{-k/2}<\gamma_{n,k}<1\quad\hbox{and}\quad [\gamma_{n,k}^{-n}p(n,n-k)]^{\frac{1}{k(n-k)}}\simeq \sqrt{n-k}.\end{equation}
\end{lemma}

\noindent {\it Proof.} Recall that
\begin{equation}\label{eq:tools-14}\gamma_{n,k}:=\omega_n^{\frac{n-k}{n}}/\omega_{n-k}.\end{equation}
Using the log-convexity of the Gamma function one can check that $e^{-k/2}<\gamma_{n,k}<1$. A proof appears in \cite[Lemma 2.1]{Koldobsky-Lifshits-2000}.

In order to give an upper bound for $p(n,n-k)$ we start from the fact
that $\omega_s=\pi^{\frac{s}{2}}/\Gamma\left (\frac{s}{2}+1\right )$ and use Stirling's approximation. Recall
that
\begin{align}\label{eq:stirling-1}
p(n,n-k) &= ((n-k)!)^{k}\frac{(n\omega_n)\cdots ((k+1)\omega_{k+1})}{((n-k)\omega_{n-k})\cdots (2\omega_2)\omega_1}\\
\nonumber &= ((n-k)!)^k\binom{n}{k}\frac{\prod_{s=k+1}^n\frac{\pi^{s/2}}{\Gamma \left(\frac{s}{2}+1\right )}}
{\prod_{s=1}^{n-k}\frac{\pi^{s/2}}{\Gamma \left(\frac{s}{2}+1\right )}}\\
\nonumber &= ((n-k)!)^k\binom{n}{k}\pi^{\frac{k(n-k)}{2}}\frac{\prod_{s=1}^{n-k}\Gamma \left(\frac{s}{2}+1\right )}{\prod_{s={k+1}}^n\Gamma \left(\frac{s}{2}+1\right )},
\end{align}
where we have used the identity
\begin{equation}\label{eq:stirling-2}\frac{1}{2}\sum_{s=k+1}^ns-\frac{1}{2}\sum_{s=1}^{n-k}s=\frac{1}{4}(n(n+1)-k(k+1)-(n-k)(n-k+1))=\frac{1}{2}k(n-k).\end{equation}
Using the estimate
\begin{equation}\label{eq:stirling-3}\left(\frac{s}{2e}\right)^{\frac{s}{2}}\sqrt{2\pi s}\ls\Gamma \left(\frac{s}{2}+1\right )\ls\left(\frac{s}{2e}\right)^{\frac{s}{2}}\sqrt{2\pi s}\,e^{\frac{1}{6s}}\ls\left(\frac{s}{2e}\right)^{\frac{s}{2}}\sqrt{2\pi s}\,e^{\frac{1}{6}}\end{equation}
we get
\begin{equation}\label{eq:stirling-4}p(n,n-k)\ls ((n-k)!)^{k}(2\pi e)^{\frac{k(n-k)}{2}}e^{\frac{n-k}{6}}\binom{n}{k}^{1/2}\frac{\prod_{s=1}^ks^{\frac{s}{2}}\prod_{s=1}^{n-k}s^{\frac{s}{2}}}
{\prod_{s=1}^ns^{\frac{s}{2}}}.\end{equation}
Let
\begin{equation}\label{eq:stirling-5}t_m=1\cdot 2^2\cdot 3^3\cdots \cdot m^m.\end{equation}
It is known that
\begin{equation}\label{eq:stirling-6}t_m\sim Am^{\frac{m^2}{2}+\frac{m}{2}+\frac{1}{12}}e^{-\frac{m^2}{4}}\end{equation}
as $m\to\infty $, where $A>0$ is an absolute constant (the Glaisher-Kinkelin constant, see e.g. \cite{Finch-book}). Note that
\begin{equation}\label{eq:stirling-7}\gamma_{n,k}^{-n}=\frac{\omega_{n-k}^n}{\omega_n^{n-k}}=\frac{\Gamma\left (\frac{n}{2}+1\right )^{n-k}}{\Gamma\left (\frac{n-k}{2}+1\right )^n}
\ls \left (\frac{n}{n-k}\right )^{\frac{n(n-k)}{2}}\frac{(\pi n)^{\frac{n-k}{2}}e^{\frac{n-k}{6}}}{(\pi (n-k))^{\frac{n}{2}}}
\ls e^{\frac{n-k}{6}}\left (\frac{n}{n-k}\right )^{\frac{(n+1)(n-k)}{2}}.\end{equation}Using the fact that $n^2=k^2+(n-k)^2+2k(n-k)$ we get
\begin{align}\label{eq:stirling-8}\gamma_{n,k}^{-\frac{n}{k(n-k)}}\left (\frac{t_kt_{n-k}}{t_n}\right )^{\frac{1}{2k(n-k)}} &\ls \frac{c_1}{\sqrt{n}}\left (\frac{k}{n}\right )^{\frac{k+1}{4(n-k)}}
\left (\frac{n-k}{n}\right )^{\frac{n-k+1}{4k}}\left (\frac{n}{n-k}\right )^{\frac{n+1}{2k}}\\
\nonumber &\ls \frac{c_1}{\sqrt{n}}\left (\frac{k}{n}\right )^{\frac{k+1}{4(n-k)}}\left (\frac{n}{n-k}\right )^{\frac{n+k+1}{4k}}\\
\nonumber &\ls \frac{c_1}{\sqrt{n}}\left (\frac{k}{n}\right )^{\frac{k+1}{4(n-k)}}\left (\frac{n}{n-k}\right )^{\frac{n-k}{2k}}\left (\frac{n}{n-k}\right )^{\frac{2k+1}{4k}}\\
\nonumber &\ls \frac{c_2}{\sqrt{n}}\cdot \frac{\sqrt{n}}{\sqrt{n-k}}=\frac{c_2}{\sqrt{n-k}}.\end{align}
Since
\begin{equation}\label{eq:stirling-9}\left [((n-k)!)^{k}(2\pi e)^{\frac{k(n-k)}{2}}e^{\frac{n-k}{6}}\binom{n}{k}^{1/2}\right ]^{\frac{1}{k(n-k)}}\ls c_3(n-k),\end{equation}
we see that
\begin{equation}\label{eq:stirling-10}[\gamma_{n,k}^{-n}p(n,n-k)]^{\frac{1}{k(n-k)}}\ls c_0\sqrt{n-k}\end{equation}
for every $1\ls k\ls n-1$, where $c_0>0$ is an absolute constant. The reverse inequality can be obtained from similar computations,
but we will not need it in the sequel. \prend

\begin{remark}\rm An alternative way to give an upper bound for $p(n,n-k)$ is to start by rewriting \eqref{eq:tools-3} in the form
\begin{equation}\label{eq:alt-1}|K|^{n-k}=p(n,n-k)\int_{G_{n,n-k}}|K\cap F|^{n}[S_{k}(K\cap F)]^{k}\,d\nu_{n,n-k}(F).\end{equation}
In particular, setting $K=B_2^n$ we see that if $k\gr 2$ then
\begin{align}\label{eq:alt-2}\omega_n^{n-k} &=p(n,n-k)\omega_{n-k}^n[S_k(B_2^{n-k})]^k\gr p(n,n-k)\omega_{n-k}^n[S_2(B_2^{n-k})]^k\\
\nonumber &\gr p(n,n-k)\omega_{n-k}^n\left (\frac{L_{B_2^{n-k}}}{\sqrt{n-k}}\right )^{k(n-k)}\gr p(n,n-k)\omega_{n-k}^n\left (\frac{c_1}{\sqrt{n-k}}\right )^{k(n-k)}\end{align}
where $c_1>0$ is an absolute constant, which implies that
\begin{equation}\label{eq:alt-3}p(n,n-k)\ls \gamma_{n,k}^n(c_0\sqrt{n-k})^{k(n-k)}.\end{equation}
where $c_0=c_1^{-1}$. For the case $k=1$ we can use the fact that $S_1(K\cap F)\gr \delta^{-(n-1)}S_2(K\cap F)$ for
every $F\in G_{n,n-1}$, and then continue as above. The final estimate is exactly the same as in Lemma \ref{lem:estimate-constant}:
\begin{equation}\label{eq:alt-4}[\gamma_{n,k}^{-n}p(n,n-k)]^{\frac{1}{k(n-k)}}\ls c_0\sqrt{n-k},\end{equation}
and this is what we use in this article. However, the proof of Lemma \ref{lem:estimate-constant} shows that this
estimate is tight for all $n$ and $k$; one cannot expect something better.
\end{remark}

For the proof of Theorem \ref{th:intro-arb-2} we will additionally use the next theorem of Dann, Paouris and Pivovarov from \cite{Dann-Paouris-Pivovarov-2015}.

\begin{theorem}[Dann-Paouris-Pivovarov]\label{th:dann-1}Let $u$ be a non-negative,
bounded integrable function on ${\mathbb R}^n$ with $\|u\|_1>0$. For every $1\ls k\ls n-1$ we have
\begin{equation}\label{eq:dann-1}\int_{G_{n,n-k}}\frac{1}{\|u|_F\|_{\infty }^k}\left (\int_Fu(x)dx\right )^nd\nu_{n,n-k}(F)
\ls \gamma_{n,k}^{-n}\left (\int_{{\mathbb R}^n}u(x)dx\right )^{n-k}.\end{equation}
\end{theorem}

The proof of this fact combines Blaschke-Petkantschin formulas with rearrangement inequalities, and
develops ideas that started in \cite{Paouris-Pivovarov-2012}.

\medskip

Finally, we use Grinberg's inequality for the dual affine quermassintegrals (introduced by Lutwak, see \cite{Lutwak-1984}
and \cite{Lutwak-1988}) of a convex body $K$ in ${\mathbb R}^n$. We
use the normalization of \cite{Dafnis-Paouris-2012}: we assume that the volume of $K$ is equal to $1$ and we set
\begin{equation}\label{eq:low-dim-slicing-1}\tilde{\Phi}_{[k]}(K)=\left(\int_{G_{n,n-k}}|K\cap F|^n
d\nu_{n,n-k}(F)\right)^{\frac{1}{kn}}\end{equation}
for every $1\ls k\ls n-1$. One can extend the definition to bounded Borel subsets of ${\mathbb R}^n$.
Grinberg \cite{Grinberg-1990} proved the following.

\begin{theorem}[Grinberg]\label{th:grinberg}Let $K$ be a compact set of volume $1$ in ${\mathbb R}^n$. For any $1\ls k\ls n-1$ and $T\in SL(n)$ we have
\begin{equation}\label{eq:grinberg-1}\tilde{\Phi }_{[k]}(K)=\tilde{\Phi }_{[k]}(T(K)).\end{equation}
Moreover,
\begin{equation}\label{eq:grinberg-2}\tilde{\Phi }_{[k]}(K)\ls \tilde{\Phi }_{[k]}(\overline{B}_2^n),\end{equation}
where $\overline{B}_2^n$ is the Euclidean ball of volume $1$.
\end{theorem}

We can use Grinberg's theorem for compact sets; this can be seen by inspection of his argument (for this more
general form see also \cite[Section 7]{Gardner-2007}). Direct computation and Lemma \ref{lem:estimate-constant} show that
\begin{equation}\label{eq:grinberg-3}\tilde{\Phi }_{[k]}(\overline{B}_2^n)
=\left (\frac{\omega_{n-k}^n}{\omega_n^{n-k}}\right )^{\frac{1}{kn}}=\gamma_{n,k}^{-1/k}\ls \sqrt{e}.\end{equation}

\section{Measure estimates for lower dimensional sections}

In this section we prove Theorem \ref{th:intro-arb-1} and Theorem \ref{th:intro-arb-2}.

\medskip

\noindent {\bf Proof of Theorem \ref{th:intro-arb-1}.} Let $\mu $ be a Borel measure with a bounded locally integrable non-negative density $g$
on ${\mathbb R}^n$. We consider a convex body $K$ in ${\mathbb R}^n$ with $0\in {\rm int}(K)$, and fix $1\ls k\ls n-1$.
Applying Lemma \ref{lem:gardner-07} with $s=n-k$ for the function $f(x_1,\ldots ,x_{n-k})=\prod_{i=1}^{n-k}g(x_i){\bf 1}_{K}(x_i)$ we get
\begin{align}\label{eq:arb-mu-1}&\mu (K)^{n-k} = \prod_{i=1}^{n-k}\int_Kg(x_i)dx=\int_{{\mathbb R}^n}\cdots \int_{{\mathbb R}^n}f(x_1,\ldots ,x_{n-k})dx_1\ldots dx_{n-k}\\
\nonumber &=p(n,n-k)\int_{G_{n,n-k}}\int_{K\cap F}\cdots\int_{K\cap F}g(x_1)\cdots g(x_{n-k})\,|{\rm conv}(0,x_1,\ldots ,x_{n-k})|^{k}dx_1\ldots dx_{n-k}\,d\nu_{n,n-k}(F)\\
\nonumber &\ls p(n,n-k)\int_{G_{n,n-k}}\int_{K\cap F}\cdots\int_{K\cap F}g(x_1)\cdots g(x_{n-k})\,|K\cap F|^{k}dx_1\ldots dx_{n-k}\,d\nu_{n,n-k}(F)\\
\nonumber &= p(n,n-k)\int_{G_{n,n-k}}|K\cap F|^k\mu (K\cap F)^{n-k}\,d\nu_{n,n-k}(F)\\
\nonumber &\ls \max_{F\in G_{n,n-k}}\mu (K\cap F)^{n-k}\cdot p(n,n-k)\int_{G_{n,n-k}}|K\cap F|^k\,d\nu_{n,n-k}(F).\end{align}
In order to estimate the last integral, note that if $\overline{K}=|K|^{-\frac{1}{n}}K$ then
\begin{align}\label{eq:arb-mu-2}\int_{G_{n,n-k}}|K\cap F|^k\,d\nu_{n,n-k}(F) &=|K|^{\frac{k(n-k)}{n}}\int_{G_{n,n-k}}|\overline{K}\cap F|^k\,d\nu_{n,n-k}(F)\\
\nonumber &\ls |K|^{\frac{k(n-k)}{n}}\left (\int_{G_{n,n-k}}|\overline{K}\cap F|^n\,d\nu_{n,n-k}(F)\right )^{\frac{k}{n}}\\
\nonumber &\ls |K|^{\frac{k(n-k)}{n}}\left (\int_{G_{n,n-k}}|\overline{B}_2^n\cap F|^n\,d\nu_{n,n-k}(F)\right )^{\frac{k}{n}}\\
\nonumber &\ls \gamma_{n,k}^{-n}|K|^{\frac{k(n-k)}{n}}
\end{align}
by Theorem \ref{th:grinberg} and \eqref{eq:grinberg-3}. Taking into account Lemma \ref{lem:estimate-constant} we see that
\begin{equation}\label{eq:arb-mu-6}\mu (K)^{n-k}\ls \left (c_0\sqrt{n-k}\right )^{k(n-k)}\max_{F\in G_{n,n-k}}\mu (K\cap F)^{n-k}|K|^{\frac{k(n-k)}{n}},\end{equation}
and the result follows. \prend

\medskip

We pass to the proof of Theorem \ref{th:intro-arb-2}. Let $\mu $ be a measure on ${\mathbb R}^n$ with a bounded locally integrable density $g$. For any $1\ls k\ls n-1$ and any convex body $K$ in ${\mathbb R}^n$ we would like to give upper and lower bounds for $\mu (K)$ in terms of the measures $\mu (K\cap F)$, $F\in G_{n,n-k}$. A lower bound can be given without any further assumption on $g$. At this point we use Theorem \ref{th:dann-1}.

\begin{proposition}\label{prop:bp-arb-1}Let $g$ be a bounded locally integrable non-negative function on ${\mathbb R}^n$ and let $\mu $ be the measure on ${\mathbb R}^n$
with density $g$. For every compact set $D$ in ${\mathbb R}^n$ we have
\begin{equation}\label{eq:bp-arb-1}\int_{G_{n,n-k}}\mu (D\cap F)^nd\nu_{n,n-k}(F)
\ls \gamma_{n,k}^{-n}\|g\|_{\infty }^k\mu (D)^{n-k}.\end{equation}
\end{proposition}

\noindent {\it Proof.} We apply Theorem \ref{th:dann-1} to the function $u=g\cdot {\bf 1}_D$. We simply observe that $\|u|_F\|_{\infty }=\|g|_{D\cap F}\|_{\infty }\ls\|g\|_{\infty }$ for all $F\in G_{n,n-k}$. Also,
\begin{equation}\label{eq:bp-arb-2}\int_Fu(x)dx=\mu (D\cap F)\quad\hbox{and}\quad \int_{{\mathbb R}^n}u(x)dx=\mu (D).\end{equation}
Then, the lemma follows from \eqref{eq:dann-1}. \prend

\medskip

We can give an upper bound if we assume that $g$ is an even log-concave function and $K$ is a symmetric convex body.

\begin{proposition}\label{prop:bp-arb-2}Let $\mu $ be a measure on ${\mathbb R}^n$ with an even log-concave density $g$. For every symmetric
convex body $K$ in ${\mathbb R}^n$ and any $1\ls k\ls n-1$ we have
\begin{equation}\label{eq:bp-arb-3}\mu (K)^{n-k}\ls p(n,n-k)\frac{(\kappa\delta kL_{n-k})^{k(n-k)}}{[(n-k)!]^{\frac{k}{2}}}\frac{1}{\|g\|_{\infty }^k}\int_{G_{n,n-k}}\mu (K\cap F)^nd\nu_{n,n-k}(F),\end{equation}
where $\kappa >0$ is the absolute constant in $\eqref{eq:Lmu}$ and $\delta >0$ is the absolute constant in Lemma $\ref{lem:arb-mu-1}$.
\end{proposition}

\noindent {\it Proof.} We start by writing
\begin{align}\label{eq:bp-arb-4}\mu (K)^{n-k} &= \prod_{i=1}^{n-k}\int_Kg(x_i)dx\\
\nonumber &=p(n,n-k)\int_{G_{n,n-k}}\int_{K\cap F}\cdots\int_{K\cap F}|{\rm conv}(0,x_1,\ldots ,x_{n-k})|^{k}\prod_{i=1}^{n-k}g(x_i)dx_1\ldots dx_{n-k}\,d\nu_{n,n-k}(F)\\
\nonumber &= p(n,n-k)\int_{G_{n,n-k}}\mu (K\cap F)^{n-k}[S_k(\mu_{K\cap F})]^k\,d\nu_{n,n-k}(F),
\end{align}
where $\mu_{K\cap F}$ is the even log-concave probability measure with density $g_{K\cap F}:=\frac{1}{\mu (K\cap F)}g\cdot {\bf 1}_{K\cap F}$.
From Lemma \ref{lem:arb-mu-1} and Lemma \ref{lem:tools-blaschke} we have
\begin{equation}\label{eq:bp-arb-5}[S_k(\mu_{K\cap F})]^k\ls (\delta k)^{k(n-k)}[S_2(\mu_{K\cap F})]^k=(\delta k)^{k(n-k)}\left (\frac{\det ({\rm Cov}(\mu_{K\cap F}))}{(n-k)!}\right )^{\frac{k}{2}}.\end{equation}
Now, since $g$ is even and log-concave we have
\begin{equation}\label{eq:bp-arb-6}\|g_{K\cap F}\|_{\infty }=\frac{g(0)}{\mu (K\cap F)}=\frac{\|g\|_{\infty }}{\mu (K\cap F)}.\end{equation}
Therefore, \eqref{eq:definition-isotropic} implies that
\begin{equation}\label{eq:bp-arb-7}\det ({\rm Cov}(\mu_{K\cap F}))=\frac{L_{\mu_{K\cap F}}^{2(n-k)}}{\|g_{K\cap F}\|_{\infty }^2}\ls \mu (K\cap F)^2\frac{(\kappa L_{n-k})^{2(n-k)}}{\|g\|_{\infty }^2},
\end{equation}
where $\kappa >0$ is the absolute constant in \eqref{eq:Lmu}. It follows that
\begin{equation}\label{eq:bp-arb-8}[S_k(\mu_{K\cap F})]^k\ls \frac{(\kappa\delta kL_{n-k})^{k(n-k)}}{[(n-k)!]^{\frac{k}{2}}}\frac{\mu (K\cap F)^k}{\|g\|_{\infty }^k}.\end{equation}
Going back to \eqref{eq:bp-arb-4} we get the result. \prend

\medskip

\noindent {\bf Proof of Theorem \ref{th:intro-arb-2}.} Combining Proposition \ref{prop:bp-arb-1} and Proposition \ref{prop:bp-arb-2} we see that
\begin{align}\mu (K)^{n-k} &\ls p(n,n-k)\frac{(\kappa\delta kL_{n-k})^{k(n-k)}}{[(n-k)!]^{\frac{k}{2}}}\frac{1}{\|g\|_{\infty }^k}\int_{G_{n,n-k}}\mu (K\cap F)^nd\nu_{n,n-k}(F)\\
\nonumber &\ls p(n,n-k)\frac{(\kappa\delta kL_{n-k})^{k(n-k)}}{[(n-k)!]^{\frac{k}{2}}}\frac{1}{\|g\|_{\infty }^k}\int_{G_{n,n-k}}\mu (D\cap F)^nd\nu_{n,n-k}(F)\\
\nonumber &\ls p(n,n-k)\frac{(\kappa\delta kL_{n-k})^{k(n-k)}}{[(n-k)!]^{\frac{k}{2}}}\frac{1}{\|g\|_{\infty }^k}\gamma_{n,k}^{-n}\|g\|_{\infty }^k\mu (D)^{n-k}\\
\nonumber &\ls \left (c_8kL_{n-k}\right )^{k(n-k)}\mu (D)^{n-k}\end{align}
for some absolute constant $c_8>0$, where in the last step we have used the estimate
\begin{equation}p(n,n-k)\ls \gamma_{n,k}^n\left (c_0\sqrt{n-k}\right )^{k(n-k)}\end{equation}
from Lemma \ref{lem:estimate-constant}. This completes the proof.\prend

\section{Volume estimates for lower dimensional sections}

In this section we collect some estimates for the volume version of the slicing problem and of the Busemann-Petty problem.
We will give two upper bounds for $\alpha_{n,k}$. These are essentially contained in the works of Dafnis and Paouris \cite{Dafnis-Paouris-2012}
and \cite{Dafnis-Paouris-2010} respectively.

\begin{proposition}\label{prop:ldslicing-1}Let $1\ls k\ls n-1$. For every centered convex body $K$ in ${\mathbb R}^n$ one has
\begin{equation}\label{eq:low-dim-slicing-4}|K|^{\frac{n-k}{n}}\ls (\overline{c}_1L_K)^k\max_{F\in G_{n,n-k}}|K\cap F|,\end{equation}
where $\overline{c}_1>0$ is an absolute constant. In particular,
\begin{equation}\label{eq:low-dim-slicing-5}\alpha_{n,k}\ls c\sqrt[4]{n},\end{equation}
where $c>0$ is an absolute constant.
\end{proposition}

\begin{proof}We may assume that the volume of $K$ is equal to $1$. It is clear that
\begin{equation}\label{eq:low-dim-slicing-6}\tilde{\Phi }_{[k]}(K)\ls \max_{F\in G_{n,n-k}}|K\cap F|^{\frac{1}{k}}.\end{equation}
By the affine invariance of $\tilde{\Phi }_{[k]}$, if $K_1$ is an isotropic image of $K$ we have
\begin{equation}\label{eq:low-dim-slicing-7}\tilde{\Phi }_{[k]}(K_1)=\tilde{\Phi }_{[k]}(K)\ls \max_{F\in G_{n,n-k}}|K\cap F|^{\frac{1}{k}}.\end{equation}
Now, we use some standard facts from the theory of isotropic convex bodies (see \cite[Chapter 5]{BGVV-book}). For every $1\ls k\ls n-1$ and $F\in G_{n,n-k}$,
the body $\overline{K_{k+1}}(\pi_{F^{\perp }}(\mu_{K_1}))$ satisfies
\begin{equation}\label{eq:low-dim-slicing-8}
|K_1\cap F|^{1/k} \gr c_1
\frac{L_{\overline{K_{k+1}}(\pi_{F^{\perp }}(\mu_{K_1}))}}{L_K},
\end{equation}
where $c_1>0$ is an absolute constant. It follows that
\begin{equation}\label{eq:low-dim-slicing-9}\tilde{\Phi}_{[k]}(K_1)L_K\gr\left(\int_{G_{n,n-k}}
(c_1L_{\overline{K_{k+1}}(\pi_{F^{\perp }}(\mu_{K_1}))})^{kn}d\nu_{n,n-k}(F)\right)^{\frac{1}{kn}}.\end{equation}
Since $L_{\overline{K_{k+1}}(\pi_{F^{\perp }}(\mu_{K_1}))}\gr c_2$ for every $F\in G_{n,n-k}$, where
$c_2>0$ is an absolute constant, we get
\begin{equation}\label{eq:low-dim-slicing-10}\tilde{\Phi}_{[k]}(K_1)L_K\gr
\left(\int_{G_{n,n-k}}(c_1L_{\overline{K_{k+1}}(\pi_F(\mu_{K_1}))})^{kn}d\nu_{n,n-k}(F)
\right)^{\frac{1}{kn}}\gr c_1c_2, \end{equation}
and the result follows from \eqref{eq:low-dim-slicing-7} with $\overline{c}_1=(c_1c_2)^{-1}$.
\end{proof}

\medskip

The next proposition provides a better bound in the case where the codimension $k$ is ``large".

\begin{proposition}\label{prop:ldslicing-2}Let $1\ls k\ls n-1$. For every centered convex body $K$ in ${\mathbb R}^n$ one has
\begin{equation}\label{eq:low-dim-slicing-11}|K|^{\frac{n-k}{n}}\ls \left ( \overline{c}_2\sqrt{n/k}\,(\log (en/k))^{\frac{3}{2}}\right )^k\max_{F\in G_{n,n-k}}|K\cap F|,\end{equation}
where $\overline{c}_2>0$ is an absolute constant. In particular,
\begin{equation}\label{eq:low-dim-slicing-12}\alpha_{n,k}\ls \overline{c}_2\sqrt{n/k}\,(\log (en/k))^{\frac{3}{2}}.\end{equation}
\end{proposition}

\begin{proof}We may assume that the volume of $K$ is equal to $1$. We consider the quantities
\begin{equation}\label{eq:low-dim-slicing-13}\tilde{W}_{[k]}(K)=\left (\int_{G_{n,n-k}} |K\cap F| d\nu_{n,k}(F)\right )^{\frac{1}{k}}\end{equation}
and
\begin{equation}\label{eq:low-dim-slicing-14}I_{-k}(K) =\left (\int_K\|x\|_2^{-k}dx\right )^{-\frac{1}{k}}.\end{equation}
Integration in polar coordinates shows that
\begin{equation}\label{eq:low-dim-slicing-15}\tilde{W}_{[k]}(K)I_{-k}(K)=
\left(\frac{(n-k)\omega_{n-k}}{n\omega_n}\right)^{1/k}=
\tilde{W}_{[k]}(\overline{B}_2^n)I_{-k}(\overline{B}_2^n)\end{equation}
and that $\left(\frac{(n-k)\omega_{n-k}}{n\omega_n}\right)^{1/k} \simeq \sqrt{n}$. It was proved in \cite{Dafnis-Paouris-2010} (see Theorem 5.2 and Lemma 5.6)
that there exists $T\in SL(n)$ such that the body $K_2=T(K)$ satisfies
\begin{equation}\label{eq:low-dim-slicing-16}I_{-k}(K_2)\ls c_1\sqrt{n}\sqrt{n/k}\,(\log (en/k))^{\frac{3}{2}}.\end{equation}
By the affine invariance of $\tilde{\Phi}_{[k]}(K)$ and by H\"{o}lder's inequality we have
\begin{equation}\label{eq:low-dim-slicing-17}\max_{F\in G_{n,n-k}}|K\cap F|^{\frac{1}{k}}\gr\tilde{\Phi}_{[k]}(K)=\tilde{\Phi}_{[k]}(K_2)
\gr \tilde{W}_{[k]}(K_2)\gr
\frac{c_2\sqrt{n}}{I_{-k}(K_2)}\end{equation} and this completes
the proof. \end{proof}

\medskip

We can also give two upper bounds for $\beta_{n,k}$. They are consequences of the next proposition.

\begin{proposition}\label{prop:bl-pet-1}Let $K$ and $D$ be two centered convex bodies in ${\mathbb R}^n$ that satisfy
\begin{equation}\label{eq:bu-pe-1}|K\cap F|\ls |D\cap F|\end{equation}
for all $F\in G_{n,n-k}$. Then,
\begin{equation}\label{eq:bu-pe-2}|K|^{\frac{n-k}{n}}\ls \left (\frac{\tilde{\Phi }_{[k]}(\overline{D})}{\tilde{\Phi }_{[k]}(\overline{K})}\right )^k\,|D|^{\frac{n-k}{n}}.\end{equation}
\end{proposition}

\noindent {\it Proof.} We have
\begin{equation}\int_{G_{n,n-k}}|K\cap F|^n\,d\nu_{n,n-k}(F)= |K|^{n-k}\int_{G_{n,n-k}}|\overline{K}\cap F|^n\,d\nu_{n,n-k}(F)=|K|^{n-k}[\tilde{\Phi }_{[k]}(\overline{K})]^{kn}\end{equation}
and
\begin{equation}\int_{G_{n,n-k}}|D\cap F|^n\,d\nu_{n,n-k}(F)= |D|^{n-k}\int_{G_{n,n-k}}|\overline{D}\cap F|^n\,d\nu_{n,n-k}(F)=|D|^{n-k}[\tilde{\Phi }_{[k]}(\overline{D})]^{kn},\end{equation}
where $\overline{A}=|A|^{-1/n}A$.

Now, from \eqref{eq:bu-pe-1} we know that
\begin{equation}\int_{G_{n,n-k}}|K\cap F|^n\,d\nu_{n,n-k}(F)\ls \int_{G_{n,n-k}}|D\cap F|^n\,d\nu_{n,n-k}(F),\end{equation}
therefore
\begin{equation}|K|^{n-k}[\tilde{\Phi }_{[k]}(\overline{K})]^{kn}\ls |D|^{n-k}[\tilde{\Phi }_{[k]}(\overline{D})]^{kn}.\end{equation}
The result follows. \prend 

\begin{remark}\rm From \eqref{eq:grinberg-3} we know that $\tilde{\Phi }_{[k]}(\overline{D})\ls\sqrt{e}$. On the other hand, in the proof
of Proposition \ref{prop:ldslicing-1} and Proposition \ref{prop:ldslicing-2} we checked that
$\tilde{\Phi}_{[k]}(\overline{K})L_K\gr c_1$ and $\tilde{\Phi}_{[k]}(\overline{K})\sqrt{n/k}\,(\log (en/k))^{\frac{3}{2}}\gr c_2$. Therefore, we always have
\begin{equation}\label{eq:bu-pe-3}\frac{\tilde{\Phi }_{[k]}(\overline{D})}{\tilde{\Phi }_{[k]}(\overline{K})}\ls c_3L_K\quad\hbox{and}\quad
\frac{\tilde{\Phi }_{[k]}(\overline{D})}{\tilde{\Phi }_{[k]}(\overline{K})}\ls c_4\sqrt{n/k}\,(\log (en/k))^{\frac{3}{2}}.\end{equation}
\end{remark}

\noindent {\bf Proof of Theorem \ref{th:intro-vol}.} The first two claims follow from Proposition \ref{prop:ldslicing-1}
and Proposition \ref{prop:ldslicing-2}. The next two are a consequence of Proposition \ref{prop:bl-pet-1} and of the previous
remark. \prend

\begin{remark}\label{rem:bl-pet-2}\rm Let us note that any upper bound for $\beta_{n,k}$ implies an upper bound for the lower dimensional
slicing problem in the symmetric case. To see this, consider a centered convex body $K$ in ${\mathbb R}^n$, fix $1\ls k\ls n-1$ and choose $r>0$ such that
\begin{equation}\label{eq:bl-pet-10}\max_{F\in G_{n,n-k}}|K\cap F|=\omega_{n-k}r^{n-k}.\end{equation}
If we set $B(r)=rB_2^n$ then we have $|K\cap F|\ls |B(r)\cap F|$ for all $F\in G_{n,n-k}$, therefore
\begin{equation}\label{eq:bl-pet-11}|K|^{\frac{n-k}{n}}\ls \big(\beta_{n,k}\big )^k|B(r)|^{\frac{n-k}{n}}=\big (\beta_{n,k}\big )^k\omega_n^{\frac{n-k}{n}}r^{n-k}.\end{equation}
It follows that
\begin{equation}\label{eq:bl-pet-12}|K|^{\frac{n-k}{n}}\ls \gamma_{n,k}\big (\beta_{n,k}\big )^k\max_{F\in G_{n,n-k}}|K\cap F|.\end{equation}
Since $\gamma_{n,k}<1$ we get:
\end{remark}

\begin{proposition}\label{prop:alpha-less-than-beta}There exists an absolute constant $c>0$ such that
\begin{equation}\label{eq:bl-pet-14}\alpha_{n,k}\ls \beta_{n,k}\end{equation}
for all $n\gr 2$ and $1\ls k\ls n-1$. \prend
\end{proposition}

\section{Concluding remarks}

\noindent {\it 6.1.} Koldobsky's approach to the slicing problem for measures is based on the following stability theorem.

\begin{theorem}[Koldobsky]\label{th:koldobsky-1}Let $1\ls k\ls n-1$ and let $K$ be a generalized $k$-intersection body in ${\mathbb R}^n$.
If $f$ is an even continuous non-negative function on $K$ such that
\begin{equation}\label{eq:kold-1}\int_{K\cap F}f(x)dx\ls\varepsilon \end{equation}
for some $\varepsilon >0$ and for all $F\in G_{n,n-k}$, then
\begin{equation}\label{eq:kold-2}\int_Kf(x)dx\ls \gamma_{n,k}\frac{n}{n-k}|K|^{\frac{k}{n}}\varepsilon .\end{equation}
\end{theorem}

The next theorem is a byproduct of our methods and provides a general stability estimate in the spirit of Theorem \ref{th:koldobsky-1}.

\begin{theorem}\label{th:stability}Let $1\ls k\ls n-1$ and let $K$ be a compact set in ${\mathbb R}^n$.
If $g$ is a locally integrable non-negative function on ${\mathbb R}^n$ such that
\begin{equation}\label{eq:kold-3}\int_{K\cap F}g(x)dx\ls\varepsilon \end{equation}
for some $\varepsilon >0$ and for all $F\in G_{n,n-k}$, then
\begin{equation}\label{eq:kold-4}\int_Kg(x)dx\ls \left (c_0\sqrt{n-k}\right )^k|K|^{\frac{k}{n}}\varepsilon .\end{equation}
\end{theorem}

\noindent {\it Proof.} Applying Lemma \ref{lem:gardner-07} with $s=n-k$ for the function $f(x_1,\ldots ,x_{n-k})=\prod_{i=1}^{n-k}g(x_i){\bf 1}_{K}(x_i)$ we get
\begin{align}\label{eq:kold-5}\prod_{i=1}^{n-k}\int_Kg(x_i)dx &\ls p(n,n-k)\int_{G_{n,n-k}}|K\cap F|^{k}\int_{K\cap F}\cdots\int_{K\cap F}g(x_1)\cdots g(x_{n-k})\,dx_1\ldots dx_{n-k}\,d\nu_{n,n-k}(F)\\
\nonumber &\ls p(n,n-k)\int_{G_{n,n-k}}|K\cap F|^k\mu (K\cap F)^{n-k}\,d\nu_{n,n-k}(F)\\
\nonumber &\ls p(n,n-k)\varepsilon^{n-k}\int_{G_{n,n-k}}|K\cap F|^{k}\,d\nu_{n,n-k}(F)\\
\nonumber &\ls \gamma_{n,k}^{-n}p(n,n-k)\varepsilon^{n-k}\ls \left (c_0\sqrt{n-k}\right )^{k(n-k)}\varepsilon^{n-k}|K|^{\frac{k(n-k)}{n}},\end{align}
using the assumption \eqref{eq:kold-6} and the bound
\begin{equation}\int_{G_{n,n-k}}|K\cap F|^{k}\,d\nu_{n,n-k}(F)\ls \gamma_{n,k}^{-n}|K|^{\frac{k(n-k)}{n}}\end{equation}
as well as Lemma \ref{lem:estimate-constant}. This shows that
\begin{equation}\label{eq:kold-6}\left (\int_Kg(x)dx\right )^{n-k}=\prod_{i=1}^{n-k}\int_Kg(x_i)dx\ls \left (c_0\sqrt{n-k}\right )^{k(n-k)}\varepsilon^{n-k}|K|^{\frac{k(n-k)}{n}},\end{equation}
and the result follows. \prend

\bigskip

\noindent {\it 6.2.} Recall that the class ${\cal BP}_k^n$ of generalized $k$-intersection bodies in ${\mathbb R}^n$, introduced by Zhang in \cite{Zhang-1996}, is the closure in the
radial metric of radial $k$-sums of finite collections of origin symmetric ellipsoids. If we define
\begin{equation}\label{eq:kold-7}{\rm ovr}(K,{\cal BP}_k^n)=\inf\left\{\left (\frac{|D|}{|K|}\right )^{1/n}:K\subseteq D,D\in {\cal BP}_k^n\right\},\end{equation}
then Theorem \ref{th:koldobsky-1} directly implies the estimate
\begin{equation}\label{eq:kold-8}\mu (K)\ls {\rm ovr}(K,{\cal BP}_k^n)^k\frac{n}{n-k}\gamma_{n,k}\max_{F\in G_{n,n-k}}\mu (K\cap F)|K|^{\frac{k}{n}}\end{equation}
for any measure $\mu $ with an even continuous density. Using \eqref{eq:kold-4} and bounds for the quantities
\begin{equation}\label{eq:kold-9}\sup_{K\in {\cal C}_n}{\rm ovr}(K,{\cal BP}_k^n),\end{equation}
Koldobsky (in some cases with Zvavitch) has obtained sharper estimates on the lower dimensional slicing problem for various classes ${\cal C}_n$ of symmetric convex bodies in ${\mathbb R}^n$:
\begin{enumerate}
\item[(i)] If $k\gr \lambda n$ for some $\lambda\in (0,1)$ then one has \eqref{eq:intro-5} for all symmetric convex bodies $K$ and all even measures
$\mu $, with a constant $\alpha $ depending only on $\lambda $ (see \cite{Koldobsky-Advances-2015}; this result employs an estimate of Koldobsky, Paouris
and Zymonopoulou for ${\rm ovr}(K,{\cal BP}_k^n)$ from \cite{Koldobsky-Paouris-Zymonopoulou-2011}).
\item[(ii)] If $K$ is an intesection body then one has \eqref{eq:intro-5} for all even measures
$\mu $, with an absolute constant $\alpha $; this was proved by Koldobsky in \cite{Koldobsky-DCG-2012} for $k=1$,
and by Koldobsky and Ma in \cite{Koldobsky-Ma-2013} for all $k$.
\item[(iii)] If $K$ is the unit ball of an $n$-dimensional subspace of $L_p$, $p>2$ then one has \eqref{eq:intro-5} for all even measures
$\mu $, with a constant $\alpha \ls cn^{\frac{1}{2}-\frac{1}{p}}$ (see \cite{Koldobsky-GAFA-2014}).
\item[(iv)] If $K$ is the unit ball of an $n$-dimensional normed space that embeds in $L_p$, $p\in (-n,2]$ then one has \eqref{eq:intro-5} for all even measures
$\mu $, with a constant depending only on $p$ (see \cite{Koldobsky-Advances-2015}).
\item[(v)] If $K$ has bounded outer volume ratio then one has \eqref{eq:intro-5} for all even measures
$\mu $, with an absolute constant $\alpha $ (see \cite{Koldobsky-Advances-2015}).
\end{enumerate}
It would be interesting to see if our method can be used for the study of special classes of convex bodies.

\bigskip

\noindent {\it 6.3.} Our proof of Theorem \ref{th:intro-arb-2} makes essential use of the log-concavity of the measure $\mu $. It was
mentioned in the introduction that Koldobsky and Zvavitch \cite{Koldobsky-Zvavitch-2015} have obtained the bound $\beta_{n,1}^{(s)}(\mu )\ls \sqrt{n}$
for every measure $\mu $ with an even continuous non-negative density. It would be interesting to see if our method can provide
this estimate, and possibly be extended to higher codimensions $k$, for more general classes of measures. It would be also
interesting to see if the symmetry assumptions on both $K$ and $\mu $ are necessary.

\bigskip

\bigskip

\footnotesize
\bibliographystyle{amsplain}

\bigskip

\bigskip

\thanks{\noindent {\bf Keywords:}  Convex bodies, isotropic position, slicing problem, Busemann-Petty problem,
Blaschke-Petkantschin formula, random simplices, dual affine quermassintegrals.}

\smallskip

\thanks{\noindent {\bf 2010 MSC:} Primary 52A23; Secondary 46B06, 52A40, 60D05.}

\bigskip

\bigskip

\noindent \textsc{Giorgos \ Chasapis}: Department of
Mathematics, University of Athens, Panepistimioupolis 157-84,
Athens, Greece.

\smallskip

\noindent \textit{E-mail:} \texttt{gchasapis@math.uoa.gr}

\bigskip

\noindent \textsc{Apostolos \ Giannopoulos}: Department of
Mathematics, University of Athens, Panepistimioupolis 157-84,
Athens, Greece.

\smallskip

\noindent \textit{E-mail:} \texttt{apgiannop@math.uoa.gr}

\bigskip

\noindent \textsc{Dimitris-Marios \ Liakopoulos}: Department of
Mathematics, University of Athens, Panepistimioupolis 157-84,
Athens, Greece.

\smallskip

\noindent \textit{E-mail:} \texttt{dimliako1@gmail.com}

\bigskip

\end{document}